\documentclass[12pt]{amsart}
\usepackage[utf8]{inputenc} 
\usepackage[T1]{fontenc}
\usepackage{color}
\usepackage{comment}
\usepackage{times}
\usepackage{cancel}
\usepackage[hidelinks]{hyperref}
\usepackage{enumerate,latexsym}
\usepackage{latexsym}
\usepackage{subcaption}

\usepackage{fullpage}

\usepackage{amsmath,amsthm,amsfonts,amssymb}
\usepackage{graphicx}
\usepackage{dsfont}

\usepackage{tikz}
\usetikzlibrary{arrows.meta}
\usetikzlibrary{knots}
\usetikzlibrary{hobby}
\usetikzlibrary{arrows,decorations.markings}

\usetikzlibrary{fadings}

\makeatletter
\def\namedlabel#1#2{\begingroup
 #2%
 \def\@currentlabel{#2}%
 \phantomsection\label{#1}\endgroup
}
\makeatother

\usepackage[lite]{amsrefs}

\renewcommand{\PrintDOI}[1]{\href{http://dx.doi.org/\detokenize{#1}}{doi: \detokenize{#1}}%
	\IfEmptyBibField{pages}{, (to appear in print)}{}}

\theoremstyle{plain}
\newtheorem*{theorem*}{Theorem}
\newtheorem*{thmex*}{Theorem~\ref{example}}
\newtheorem*{thmasymp*}{Theorem~\ref{thmAsymp}}
\newtheorem{theorem}{Theorem}[section]

\newtheorem{corollary}[theorem]{Corollary}

\newtheorem{lemma}[theorem]{Lemma}

\newtheorem{remark}[theorem]{Remark}

\newtheorem{example}[theorem]{Example}

\theoremstyle{definition}
\newtheorem{definition}[theorem]{Definition}

\newcommand{\ben}{\begin{enumerate}}
\newcommand{\een}{\end{enumerate}}

\newcommand{\ed}{\end{document}}

\definecolor{rrr}{rgb}{.9,0,.1}

\definecolor{rr}{rgb}{.8,0,.3}

\newcommand{\tr}{\triangleright}
\renewcommand\hom{\operatorname{Hom}}
\newcommand\End{\operatorname{End}}

\usepackage{pdfsync}
\graphicspath{ {./images/} }

\title[Pointed Quandle Coloring Quivers]{Pointed Quandle Coloring Quivers of Linkoids}

\author[J. Ceniceros]{Jose Ceniceros}
\address{Hamilton College, Clinton, NY, USA}
\email{jcenicer@hamilton.edu}
\author[M. Klivans]{Max Klivans}
\address{Hamilton College, Clinton, NY, USA}
\email{mklivans@hamilton.edu}

\begin{document}

\maketitle

\begin{abstract}
We enhance the pointed quandle counting invariant of linkoids through the use of quivers analogously to quandle coloring quivers. This allows us to generalize the in-degree polynomial invariant of links to linkoids. Additionally, we introduce a new linkoid invariant, which we call the in-degree quiver polynomial matrix. Lastly, we study the pointed quandle coloring quivers of linkoids of $(p,2)$-torus type with respect to pointed dihedral quandles.   

\end{abstract}

\parbox{5.5in} {\textsc{Keywords:} quandles, pointed quandles, knotoids, linkoids, quivers

                \smallskip
                
                }

\section{Introduction}\label{intro}
Knotoids were introduced by Turaev in \cite{Tu}. Knotoids are knot diagrams with loose ends that may be on different regions of the diagram. Knotoids can be considered as generalizations of various knotted objects like tangles, where the endpoints can be positioned in any region of the tangle's complement, with some restrictions on the movement of the endpoints to the starting region. Additionally, knotoids are an extension of classical knot theory and thus have recently become the subject of much research. Specifically, several articles have been written exploring the properties of knotoids as well as defining invariants of knotoids, see \cites{BBA, MV, GG, GK1, GK2}. Besides the interest from low-dimensional topologists, researchers from the area of biology and chemistry have also become interested in knotoids as they provide a formal mathematical foundation for studying the entanglement of proteins, see \cite{DDFS, GGLDSL, BG, GDS}.

The quandle, which is an algebraic structure, was introduced by Joyce \cite{Joyce} and Matveev \cite{Matveev} independently. Since then, quandles have attracted attention from topologists. The interest from topologists comes from the fact that the quandle axioms capture the Reidemeister moves from knot theory. Furthermore, Joyce and Matveev independently proved that the fundamental quandle is a complete invariant of knots and links up to mirror image and orientation reversal. Although the fundamental quandle is a powerful invariant, in a way, this is trading a difficult knot theory question for a difficult algebra question. The main difficulty when considering the fundamental quandle of a knot is that comparing two presentations of two fundamental quandles is a significantly difficult problem. Therefore, much work has been done to define computable invariants derived from the fundamental quandle. Specifically, the quandle counting invariant and enhancements to the quandle counting invariant. See \cite{Joyce, CJKLS}.

A specific enhancement of the quandle counting invariant was defined by choosing a subset of the set of endomorphism of a quandle and defining a quiver-valued invariant of classical knots and links, this invariant is the quandle coloring quiver \cite{CN}. Since then several papers have been written investigating the properties and generalizations of this enhancement, see \cite{BaCa, CN1, EJL, CCN }.

In \cite{GuPf, Pf}, G\"ug\"umc\"u and Pflume introduced the fundamental quandle of linkoids, and some of the basic properties were studied. They were able to prove that the fundamental quandle is an invariant of linkoids, but it was also invariant under the so-called forbidden moves. If allowed, the forbidden moves would allow the movement of endpoints under or over arcs and thus may transform the linkoid into a non-equivalent linkoid. G\"ug\"umc\"u and Pflume addressed this issue by defining a generalization of quandles called $n$-pointed quandles. Pointed quandles were used to define the fundamental pointed quandle of linkoids, as well as the pointed quandle coloring invariant and quandle coloring matrix. Furthermore, they were able to show the effectiveness of their new invariant by providing examples of knotoids that are not distinguished using quandles but are distinguished by pointed quandles.

In this article, we generalize the quandle coloring quiver to the case of pointed quandles to obtain the pointed quandle coloring quiver. We introduce two invariants of linkoids derived from the pointed quandle coloring quiver: one in the form of a polynomial and the other in the form of a matrix. We demonstrate the strength of these new invariants for linkoids by providing examples of linkoids that are not distinguished by either the pointed quandle counting invariant or the quandle matrix invariant but are distinguished by the new invariants. Furthermore, we study the pointed quandle quiver of a family of linkoids.

This article is organized as follows. In Section~\ref{linkoids}, we will go over the basics of knotoids and linkoids, including their Reidemeister moves and the forbidden moves. In Section~\ref{quandles}, we will review the definition of a quandle, discuss its shortcomings when studying linkoids, and recall pointed quandles. In Section~\ref{pointedquiver}, we will introduce a new invariant of linkoids called the pointed quandle coloring quiver. In Section~\ref{indegree}, we define the in-degree polynomial invariant as well as the in-degree polynomial matrix invariant of linkoids. Additionally, we provide examples to demonstrate that the in-degree polynomial and the in-degree polynomial matrix are enhancements of the quandle counting invariant and the quandle counting matrix invariant. Lastly, in Section~\ref{torus-type}, we study the pointed quandle coloring quivers of linkoids of $\mathcal{T}(p,2)$-type.

\section{Linkoids}\label{linkoids}

In this section, we give a brief overview of the basic definitions and results of knotoids and linkoids. For further details, please refer to \cite{Tu,GK1, GK2,GuPf, Pf}.

\begin{definition}\label{linkoid}\cite{GuPf, Pf}
    For any $n\ge 0$, an oriented \emph{$n$-linkoid diagram} in $S^2$ is a generic immersion of $n$ unit intervals $[0, 1]$ and a number of oriented unit circles $S^1$ into $S^2$ with finitely many transverse double points. Each double point is endowed with over/under-crossing data. A component is \emph{open} if it is the image of $[0,1]$ and a component is \emph{closed} if it is the image of $S^1$. In an open component, the image of $0$ is called the \emph{leg} and the image of $1$ is called the \emph{head} of the component.

    In particular, a $1$-linkoid diagram with no closed components is called a \emph{knotoid} diagram, a $0$-linkoid diagram with exactly one closed component is a \emph{knot diagram}, and a $0$-linkoid diagram with at least one closed component is a \emph{link diagram}. A \emph{full linkoid diagram} is an $n$-linkoid diagram with no closed components.
\end{definition}

We note that, in general, linkoids can be defined as generic immersions into any orientable surface. Specifically, in the case when the surface is $\mathbb{R}^2$, the linkoids are called \emph{planar linkoids}. In this article, we will focus on linkoids in $S^2$, which are called \emph{spherical linkoids}. Since we focus on spherical linkoids, we will refer to them as linkoids.

\begin{figure}[h!]
    \centering
\includegraphics{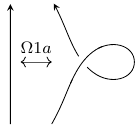}\hspace{2cm}
\includegraphics{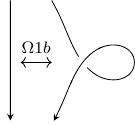}\\
\includegraphics{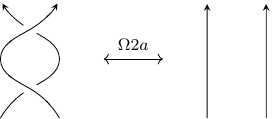}\\
\includegraphics{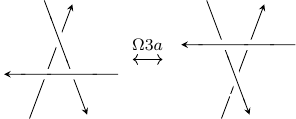}
\caption{A minimal generating set of oriented Reidemeister moves.}
\label{fig:rmoves}
\end{figure}

\begin{definition}
    Two linkoid diagrams $L$, $L'$ in $S^2$ are equivalent if one can be moved into the other by a finite sequence of local moves known as the oriented Reidemeister moves depicted in Figure~\ref{fig:rmoves} and isotopy of $S^2$. We denote the equivalence of linkoid diagrams by $L \sim L'$. The equivalence classes of these diagrams are called \emph{linkoids}.
\end{definition}

We note that in \cite{P}, the four moves in Figure~\ref{fig:rmoves} were shown to form a minimal generating set of oriented Reidemeister moves. Additionally, linkoid isotopy in $S^2$ involves an additional move that allows pulling an arc so that it crosses either of the ``poles'' of $S^2$ as described in \cite{GL}. This is represented by the S-move in Figure~\ref{fig:smove}.  Lastly, the Reidemeister moves applied to linkoid diagrams are not allowed to involve any endpoints of the linkoid. More specifically, endpoints are not allowed to be moved over or under a strand. Therefore, the moves in Figure~\ref{fig:forbidden} are not allowed. These moves are called the \emph{forbidden moves}. If the forbidden moves were allowed, then any linkoid would be equivalent to the trivial linkoid.

\begin{figure}[h!]
    \centering
\includegraphics{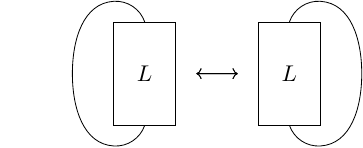}
\caption{The S-move.}
\label{fig:smove}
\end{figure}

\begin{figure}[h!]
    \centering
\includegraphics{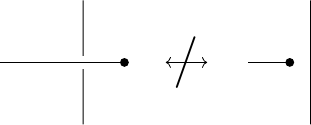}\hspace{1cm}
\includegraphics{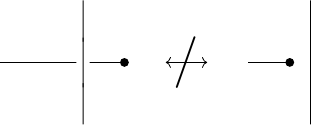}
\caption{The over and under forbidden moves.}
\label{fig:forbidden}
\end{figure}

The following definition will be important in Section~\ref{torus-type}.

\begin{definition}
    A $1$-linkoid diagram is called a \emph{link-type} if the endpoints of the open component lie in the same region of the diagram. Specifically, if we have a knotoid diagram with two end points lying in the same region, it is called a \emph{knot-type} knotoid.
\end{definition}

The diagram shown in Figure~\ref{fig:T(4,2)} depicts the 1-linkoid of $\mathcal{T}(4,2)$-type. We observe that the endpoints labeled as $x_1$ and $x_5$ are in the same region of the diagram. Additionally, if these two endpoints were connected, the resulting diagram would represent the $(4,2)$-torus link, which is commonly denoted by $\mathcal{T}(4,2)$.

\begin{figure}[h]
    \centering
\includegraphics{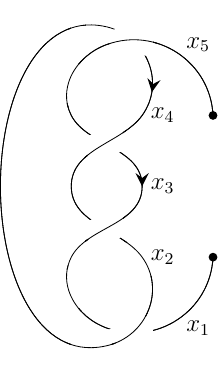}
    \caption{ A 1-linkoid diagram of $\mathcal{T}(4,2)$-type}
    \label{fig:T(4,2)}
\end{figure}

\section{Quandles and Pointed Quandles}\label{quandles}

In this section, we review the basics of quandles. More detailed information on this topic can be found in \cite{EN,Joyce, Matveev}. We will also review pointed quandles. For more details on this topic, refer to \cite{GuPf, Pf}.
\begin{definition}\label{quandle}
    A \emph{quandle} is an ordered pair $(X, \tr)$, where $X$ is a set and $\tr:X\times X\to X$ is a binary operation on $X$ such that
    \begin{enumerate}
        \item for all $x\in X$, $x\tr x = x$,
        \item for all $y\in Y$, the function $\beta_y:X\to X$, defined by $\beta_y(x)=x\tr y$, is a bijection, and 
        \item for all $x,y,z\in X$, $(x\tr y)\tr z = (x\tr z) \tr (y\tr z)$.
    \end{enumerate}
    We will denote the quandle $(X,\tr)$ by $X$ if the operation $\tr$ is clear.
\end{definition}

\begin{figure}[h!]
    \centering
\includegraphics{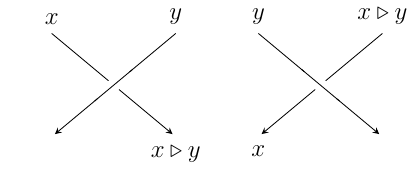}
    \caption{Quandle relations at a positive and negative crossing.}
    \label{fig:rule}
\end{figure}

The axioms of a quandle correspond respectively to the Reidemeister moves of types I, II, and III following relations in Figure~\ref{fig:rule}. To see more details about the relationship between quandles and the Reidemeister moves, see \cite{EN}. The following are typical examples of quandles: 

\begin{example}\label{trivial}
    For any set $X$, the \emph{trivial quandle} is the quandle $(X, \tr)$, where, for each $x,y\in X$, $x\tr y= x$.
\end{example}

\begin{example}\label{def:dihedral}
    The set of integers modulo $n$, denoted by $\mathbb{Z}_n$, is a quandle with the operation $x \tr y = 2y - x$. The quandle $(\mathbb{Z}_n,\tr)$ is called the dihedral quandle of order $n$.
\end{example}

In order to simplify the notation, we will use $\mathbb{Z}_n$ to denote the dihedral quandle of order $n$. Otherwise, we will specify that we refer to the set of integers modulo $n$.


In \cite{Joyce, Matveev}, the quandle was used to define the following quandle of a link. For every oriented link $L$ with diagram $D$, there is an associated fundamental quandle of $L$ denoted by $Q(L)$. The fundamental quandle of $L$ is defined as the free quandle on the set of arcs in the diagram $D$ modulo the equivalence relations generated by the crossing relations from the diagram $D$ of $L$. In \cite{GuPf,Pf}, the fundamental quandle was extended to linkoids in the following way. 

\begin{definition}\label{fundquandle}
    Let $L$ be an oriented linkoid diagram and $A(L)$ be the set of arcs in $L$. Then the \emph{fundamental quandle} of $L$ is 
    $$Q(L)= \langle x\in A(L)\mid r_\tau\text{ for all crossings $\tau$ in $L$}\rangle,$$
    the quandle generated by the arcs of $L$ modulo the relations given by each crossing as in Figure~\ref{fig:rule}.
\end{definition}

The fundamental quandle was independently introduced by Joyce in \cite{Joyce} and Matveev in \cite{Matveev}. Also, both Joyce and Matveev proved that the fundamental quandle can distinguish all oriented knots up to mirror image with reverse orientation. It was shown in \cite{GuPf,Pf} that the fundamental quandle is an invariant of linkoids. This means that the fundamental quandle only depends on the linkoid and not the choice of linkoid diagram.
On the other hand, by Lemma 3.4 in \cite{GuPf,Pf}, the fundamental quandle of a linkoid is invariant under the forbidden moves, so the fundamental quandle alone is much less helpful in the classification of linkoids. This motivated the following definitions.

\begin{definition}\cite{GuPf,Pf}
    An \emph{$n$-pointed quandle} $(X,x_1,\dots,x_n)$ is an ordered tuple consisting of a quandle $X$ and $n$ elements $x_1,\dots,x_n\in X$, the \emph{basepoints} of the pointed quandle. 
\end{definition}

\begin{definition} \cite{GuPf, Pf}An \emph{ordered $n$-linkoid diagram} is an $n$-linkoid diagram where the open components are enumerated.
\end{definition}

\begin{definition}\cite{GuPf, Pf}
    Let $L$ be an ordered $n$-linkoid diagram. The \emph{fundamental pointed quandle of $L$} is the $2n$-pointed quandle
    $$P(L) = (Q(L),l_1,h_1,\dots,l_n,h_n),$$
    where $Q(L)$ is the fundamental quandle of $L$ and, for each $i\in\{1,\dots,n\}$, $l_i$ and $h_i$ are the labels corresponding to the leg and head of the $i$-th open component of $L$.
\end{definition}

In \cite{GuPf, Pf}, G\"ug\"umc\"u and Pflume were able to show that the fundamental pointed quandle of $L$ is invariant under Reidemeister moves and the S-move. Thus, the fundamental pointed quandle is an invariant of linkoids and only depends on the linkoid and not on the choice of linkoid diagram. 

In order to define computable invariants derived from fundamental pointed quandles of linkoids, it is useful to consider homomorphisms of pointed quandles. First recall the definition of quandle homomorphism.  Let $(X,\tr_X)$ and $(Y,\tr_Y)$ be quandles and let $f:X\to Y$ be a function. Then $f$ is a \emph{quandle homomorphism} if $f(x_1\tr_Xx_2)=f(x_1)\tr_Yf(x_2)$ for all $x_1,x_2\in X$. We denote the set of all such functions $\hom(X,Y)$. The following was introduced in \cite{GuPf, Pf}, but we reformulate the definition here. 

\begin{definition}
    Let $\mathcal X=(X,x_1,\dots,x_n)$ and $\mathcal Y=(Y, y_1,\dots,y_n)$ be two $n$-pointed quandles and let $f:X\to Y$ be a quandle homomorphism. Then $f$ is a \emph{pointed quandle homomorphism} if $f(x_i)=y_i$ for all $i\in\{1,\dots,n\}$.
\end{definition}

    In what follows, we will use both unpointed quandles and pointed quandles. We will use $X$ to denote unpointed quandles and $\mathcal{X}$ to denote pointed quandles. 
    We denote the set of all such pointed homomorphisms from $\mathcal{X}$ to $\mathcal{Y}$ by $\hom(\mathcal X,\mathcal Y)$. The pointed quandle homomorphism set is differentiated from the set of all quandle homomorphisms by the presence of basepoints. Note that $\hom(\mathcal X,\mathcal Y)\subseteq \hom(X,Y)$. Lastly, we will use $\End(\mathcal{X})$ to denote the set of pointed quandle endomorphisms of $\mathcal{X}$.

\begin{definition}

    Let $P(L)$ be the fundamental pointed quandle of an $n$-linkoid $L$ and let $\mathcal X=(X,x_1,x_2,\dots,x_{2n})$ be a finite $2n$-pointed quandle. The \emph{pointed quandle counting invariant} of $L$ with respect to $\mathcal X$ is \[\Phi_{\mathcal X}^{\mathbb Z}(L)=|\hom(P(L),\mathcal X)|.\]
    \end{definition}

\begin{definition}
    
    If $L$ is a $1$-linkoid, then the \emph{quandle counting matrix} $\Phi_X^{M_k}(L)$ of $L$ with respect to the finite unpointed quandle $X=\{x_1,x_2,\dots,x_k\}$ is the $k\times k$ matrix whose $(i,j)$-th entry is
    $$\left(\Phi_X^{M_k}(L)\right)_{i,j}=|\hom(P(L),(X,x_i,x_j))|.$$
\end{definition}

\begin{remark}
       We will refer to $\hom(P(L),\mathcal{X})$ as the coloring set of the $n$-linkoid $L$ by $\mathcal{X}$. Furthermore, each $\alpha \in \hom(P(L),\mathcal{X})$ is an $\mathcal{X}$-coloring of $L$. 
\end{remark}

In \cite{GuPf, Pf}, it is established that the pointed quandle counting invariant and the quandle counting matrix are invariants of linkoids. Also, see \cite{GuPf, Pf} for examples of the pointed quandle counting invariant and the quandle counting matrix.

\section{Pointed Quandle Coloring Quivers} \label{pointedquiver}

\begin{definition}\label{def:multigraph}
    A \emph{directed multigraph} is an ordered pair $G=(V,w)$, where $V$ is any finite set, and $w:V\times V\to \mathbb N\cup \{0\}$ is a function. The elements of $V$ are called \emph{vertices}, and for all $u,v\in V$, if $w(u,v)=n$, then there are $n$ \emph{arcs} from $u$ to $v$ in $G$. 
    If there are multiple graphs in question, then we denote $V=V(G)$ and $w=w_G$.
\end{definition}

Directed multigraphs lend themselves naturally to drawings in the plane. Represent each vertex with a dot, and, for each $u,v\in V$, draw $w(u,v)$ arrows from the dot representing $u$ to the dot representing $v$.

\begin{definition}
    Let $P(L)$ be the fundamental pointed quandle of an $n$-linkoid diagram $L$, let $\mathcal X$ be a finite $2n$-pointed quandle, and let $S\subset \End(\mathcal X)$. The \emph{pointed quandle coloring quiver of $L$ with respect to $\mathcal{X}$} is the directed multigraph $\mathcal Q_{\mathcal X}^S(L)=(V,w)$, where $V=\hom(P(L),\mathcal X)$, and, for all $\alpha,\beta\in V$, $w(\alpha,\beta)=|\{\varphi\in\End(\mathcal X)\mid \varphi\circ\alpha=\beta\}|$. In the case when $S = \End{\mathcal{X}}$ we will call this
    the \emph{full pointed quandle coloring quiver of $L$ with respect to $\mathcal{X}$} and will be denoted by $\mathcal Q_{\mathcal{X}}(L)$.
\end{definition}

That is, each vertex in $\mathcal Q_{\mathcal X}^S(L)$ is an $\mathcal X$-coloring of $L$ and, for all $\alpha,\beta\in V(\mathcal Q_{\mathcal X}^S(L))$, there is an arc from $\alpha$ to $\beta$ for each $\varphi\in S$ such that $\varphi\circ \alpha=\beta$. The coloring set of a linkoid by a pointed quandle may be empty. This is unlike the case of classical knots and links, where the trivial coloring is always valid with respect to any quandle. However, this is not always possible for linkoids and pointed quandles. See Example~\ref{ex:effective} for examples of linkoids with no valid colorings by several pointed quandles. In such a situation, the associated quiver has no vertices and no edges.

\begin{theorem}
    Let $\mathcal{X}$ be a finite $2n$-pointed quandle, $S\subseteq \End(\mathcal{X})$ and $L$ is a oriented $n$-linkoid. Then the quiver $\mathcal{Q}_\mathcal{X}^S(L)$ is an invariant of $L$.
\end{theorem}
\begin{proof}
For every $\mathcal{X}$-coloring $\alpha \in \textup{Hom}(P(L), \mathcal{X})$ and $\varphi: \mathcal{X} \rightarrow \mathcal{X}$, the conditions needed for $\beta = \varphi \circ \alpha$ to be an $\mathcal{X}$-coloring of $L$ are exactly the conditions needed for $\varphi$ to be a pointed quandle endomorphism. Since the pointed quiver $\mathcal{Q}_{\mathcal{X}}^S(L)$ is determined up to isomorphism by $\mathcal{X}$ and $\textup{Hom}(P(L), \mathcal{X})$, the pointed quiver is an invariant of $L$.
\end{proof}

\begin{corollary}\label{invariant}
    Any invariant of directed multigraphs applied to $\mathcal{Q}_\mathcal{X}^S(L)$ defines an invariant of oriented $n$-linkoids.
\end{corollary}

\section{The In-degree Quiver Polynomial and Matrix}\label{indegree}

In their work \cite{CN}, Cho and Nelson discussed how the number of edges out of a vertex in a quandle coloring quiver corresponds to the cardinality of $S \subseteq \textup{End}(X)$.  They also remarked that the number of edges into a vertex $v$, called the \emph{in-degree} of the vertex and denoted by $\text{deg}^+(v)$, may be different. Cho and Nelson used this to define the in-degree quiver polynomial of a link. The following definition generalizes this polynomial invariant to $n$-linkoids. Additionally, we introduce the in-degree quiver polynomial matrix of 1-linkoids. 

\begin{definition}
Let $\mathcal{X}$ be a finite $2n$-pointed quandle, $S \subset \textup{End}(\mathcal{X})$ a set of $2n$-pointed quandle endomorphisms, $L$ an oriented $n$-linkoid and $\mathcal{Q}_\mathcal{X}^S(L)$ the associated pointed quandle coloring quiver of $L$ whose set of vertices is $V = \textup{Hom}(P(L), \mathcal{X})$. Then the \emph{in-degree quiver polynomial of $L$} with respect to $\mathcal{X}$ is 

\[\Phi_{\mathcal{X}}^{\textup{deg}^+, S}(L) = \sum_{f \in V} u^{deg^+(f)}.\]
In the case when $S = \End{\mathcal{X}}$ we will call this the \emph{full in-degree quiver polynomial of $L$ with respect to $\mathcal{X}$} and will be denoted by $\Phi_{\mathcal{X}}^{\textup{deg}^+}(L)$.

\end{definition}

\begin{definition}
Let $X= \{x_1, x_2, \dots, x_k\}$ be a finite quandle, so, for each $i,j\in\{1,\dots,k\}$, $(X,x_i,x_j)$ is a $2$-pointed quandle and $S_{i,j}\subset \End{(X,x_i,x_j)}$. Let $\mathbf{S}=\{S_{i,j}\mid i,j\in\{1,\dots,k\}\}$. 
Let $L$ be an oriented $1$- linkoid. Then the \emph{in-degree quiver polynomial matrix $\Phi_{X,\mathbf{S}}^{M_k,\textup{deg}^+}(L)$ of $L$ with respect to $X$} is the $k \times k$ matrix whose $(i,j)$-th entry is
\[\left(\Phi_{X,\mathbf{S}}^{M_k,\textup{deg}^+}(L) \right)_{i,j} = \Phi_{(X,x_i,x_j)}^{\textup{deg}^+, S_{i,j}}(L).
\]
If $S_{i,j} = \End{(X,x_i,x_j)}$ for all $i,j\in\{1,\dots,k\}$, then we will call this the \emph{full in-degree quiver polynomial matrix of $L$ with respect to $X$} and will denote it by $\Phi_{X}^{M_k,\textup{deg}^+}(L)$.

\end{definition}

By construction and Corollary~\ref{invariant}, we obtain the following two results.

\begin{corollary}
    The in-degree quiver polynomials are invariants of $n$-linkoids in the case of $2n$-pointed quandles. 
\end{corollary}
\begin{corollary}
    The in-degree quiver polynomial matrices are invariants of $1$-linkoids in the case of quandles. 
\end{corollary}

The counting invariant $\Phi_\mathcal{X}^\mathbb{Z}(L)$ and the counting matrix $\Phi_X^{M_k}(L)$ associated with $\mathcal{X}$ and $X$ respectively are computable and effective invariants of linkoids. However, a set is more than its cardinality, and the following examples show that the in-degree quiver polynomial and the in-degree quiver polynomial matrix extract additional information from the coloring sets. 
\begin{example}\label{ex:effective}
    Let $X$ be the quandle of cardinality $4$ with the following operation table,
\[
\begin{tabular}{l|cccc}
 $\tr$ & 0 & 1 & 2 &3\\
 \hline
 0& 0 & 2 & 0 & 2 \\
 1& 3 & 1 & 3 & 1 \\
 2& 2 & 0 & 2 & 0 \\
 3& 1 & 3 & 1 & 3
\end{tabular}.
\]
\begin{figure}[h!]
    \centering
\includegraphics{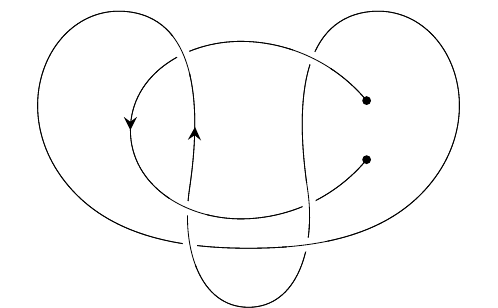}
    \caption{Digram $D_1$ of the 1-linkoid $L_1$.}
    \label{fig:16a1}
\end{figure}

\begin{figure}[h!]
    \centering
\includegraphics{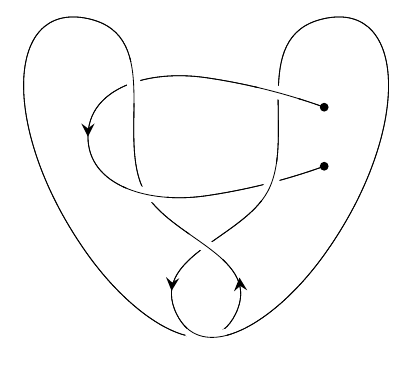}
    \caption{Diagram $D_2$ of the 1-linkoid $L_2$}
    \label{fig:16a5}
\end{figure}
    We will consider the following two oriented linkoids $L_1$ and $L_2$ with diagrams $D_1$ and $D_2$, see Figure~\ref{fig:16a1} and \ref{fig:16a5}. The two linkoids have equal quandle counting matrices with respect to $X$
\[ \Phi_X^{M_4} (L_1) =\begin{bmatrix}
4 & 0 & 0 & 0\\
0 & 4 & 0 & 0\\
0 & 0 & 4 & 0\\
0 & 0 & 0 & 4
\end{bmatrix} = \Phi_X^{M_4} (L_2),\]
 but the two linkoids are distinguished by their full in-degree quiver polynomial matrix,

\begin{eqnarray*}
\Phi_X^{M_4,\textup{deg}^+} (L_1) &=& \begin{bmatrix}
u^8+u^4+2u^2 & 0 & 0 & 0\\
0 & u^8+u^4+2u^2 & 0 & 0\\
0 & 0 & u^8+u^4+2u^2 & 0\\
0 & 0 & 0 & u^8+u^4+2u^2 
\end{bmatrix}\neq \\
\Phi_X^{M_4,\textup{deg}^+} (L_2) &=& \begin{bmatrix}
u^{10}+3u^2 & 0 & 0 & 0\\
0 & u^{10}+3u^2 & 0 & 0\\
0 & 0 & u^{10}+3u^2 & 0\\
0 & 0 & 0 & u^{10}+3u^2 
\end{bmatrix}.
\end{eqnarray*}

Additionally, from the full in-degree quiver polynomial matrix, we can see that we only have to consider a pointed quandle of the form $(X, i,i)$. For example, consider the pointed quandle $\mathcal{X}=(X, 0,0)$ where $X$ is the quandle defined above. The pointed quandle counting invariant of $L_1$ and $L_2$ are equal,
$\Phi_{\mathcal{X}}^\mathbb{Z}(L_1) = 4 = \Phi_{\mathcal{X}}^\mathbb{Z}(L_2)$, but their full pointed quandle coloring quivers distinguish them, see Figure~\ref{fig:16a116a5Quiver}. Specifically the full in-degree quiver polynomial with respect to $\mathcal{X}=(X, 0,0)$ distinguishes the two linkoids,
\[\Phi_{\mathcal{X}}^{\textup{deg}^+}(L_1) = u^8+u^4+2u^2 \neq u^{10}+3u^2 =\Phi_{\mathcal{X}}^{\textup{deg}^+}(L_2)\]

\begin{figure}[h]
    \centering
    \includegraphics[width=0.4\linewidth]{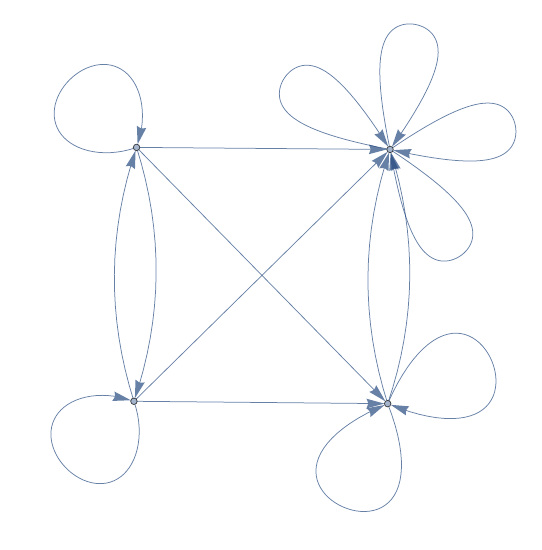}
    \includegraphics[width=0.4\linewidth]{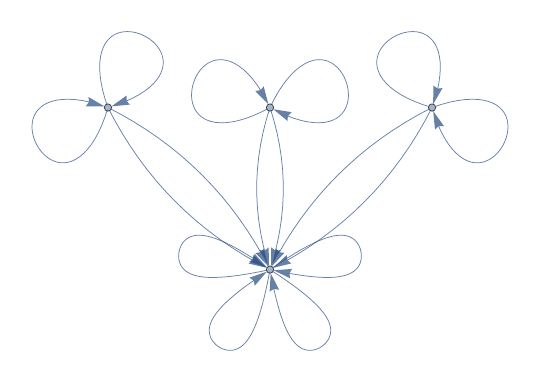}
    \caption{Full pointed quandle coloring quivers. On the left, we have $\mathcal Q_{\mathcal X}(L_1)$ and on the right we have $\mathcal Q_{\mathcal X}(L_2)$.}
    \label{fig:16a116a5Quiver}
\end{figure}
\end{example}

\section{Pointed Quandle Coloring Quivers of Linkoids of $\mathcal{T}(p,2)$-type}\label{torus-type}

In \cite{BaCa}, Basi and Caprau computed the (unpointed) quandle coloring quiver of the $\mathcal T(p,2)$ torus link with respect to the dihedral quandle, in \cite{ZhLi}, Liu and Zhou did the same for the $\mathcal T(p,3)$ torus link, and in \cite{EJL}, Elhamdadi, Jones, and Liu generalized Basi, Caprau, Liu, and Zhou's work to the $\mathcal T(p,q)$ torus link. In this section, we study the pointed quandle coloring quiver of the $1$-linkoid of $\mathcal T(p,2)$-type which we denote 
$\widetilde{\mathcal{T}(p,2)}$. In particular, we will compute $\mathcal Q_{(\mathbb{Z}_n,y_1,y_2)}(\widetilde{\mathcal T(p,2)})$ when 
$y_1\not\equiv y_2\bmod n$, $\gcd(p,n)=1$, or $\gcd(p,n)$ is prime. In what follows, we consider the diagram for $\widetilde{\mathcal{T}(p,2)}$ in Figure~\ref{fig:T(p,2)}. The diagram consists of $p$ crossings and $p+1$ arcs labeled $\lbrace x_1, x_2, \dots, x_{p+1}\rbrace$. Additionally, $\mathbb Z_n$ will always denote the dihedral quandle defined in Example~\ref{def:dihedral}, and $\mathcal Z$ will always denote a $2$-pointed dihedral quandle with one repeated basepoint. That is, $\mathcal Z=(\mathbb Z_n,y,y)$ for some $y\in \mathbb Z_n$.

\begin{remark}\label{rem:Q(T(p,2)}
\begin{figure}[h!]
    \centering
\includegraphics{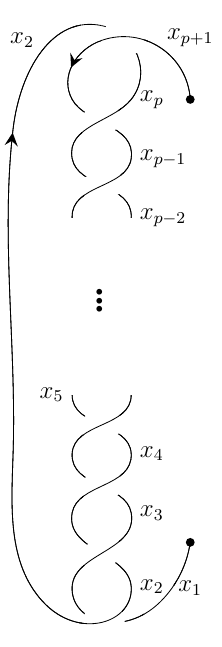}
    \caption{A 1-linkoid diagram of $\widetilde{\mathcal T(p,2)}$.}
    \label{fig:T(p,2)}
\end{figure}

From Figure~\ref{fig:T(p,2)}, we obtain the following presentation for the fundamental quandle of $\widetilde{\mathcal T(p,2)}$: 
$$Q(\widetilde{\mathcal T(p, 2)}) = \langle x_1,x_2,\dots, x_p, x_{p+1}\mid x_p=x_2\tr x_{p+1} \text{ and }x_i=x_{i+2}\tr x_{i+1} \text{ for all }1\le i\le p-1\rangle.$$
Furthermore, the fundamental pointed quandle of $\widetilde{\mathcal T(p, 2)}$ is $P(\widetilde{\mathcal T(p,2)})=(Q(\widetilde{\mathcal T(p,2)}), x_1, x_{p+1})$.

\end{remark}

We note that from the top crossing of the diagram, we will obtain the relation $x_2 \tr x_{p+1} = x_p$, and the remaining $p-1$ crossings follow the descending pattern given by $x_{i+2} \tr x_{i+1} = x_i$. 
Since $P(\widetilde{\mathcal{T}(p,2)})$ has two basepoints, we only consider $2$-pointed dihedral quandles.
The following lemma derives some facts about the colorings of the $\widetilde{\mathcal T(p,2)}$ linkoids with respect to $\mathbb{Z}_n$ from the relations of the fundamental pointed quandle in Remark~\ref{rem:Q(T(p,2)}.

\begin{lemma}\label{lem:Q(T(p,2))homprops}
    Let $p,n\in\mathbb{N}$. If $\varphi\in \hom(Q(\widetilde{\mathcal T(p,2)},\mathbb Z_n)$, then $p\varphi(x_2)\equiv p\varphi(x_{p+1})\bmod n$ and $\varphi(x_1)\equiv\varphi(x_{p+1})\bmod n$.
\end{lemma}

\begin{proof}
    Let $\varphi\in \hom(Q(\widetilde{\mathcal T(p,2)}), \mathbb Z_n)$ and$i\in\{1,\dots, p-1\}$. By the relations from Remark~\ref{rem:Q(T(p,2)} and since $\varphi$ is a quandle homomorphism, 
    \begin{align*}
        \varphi(x_i) &\equiv \varphi(x_{i+2}\tr x_{i+1})\\
        &\equiv \varphi(x_{i+2})\tr \varphi(x_{i+1})\\
        &\equiv 2\varphi(x_{i+1})-\varphi(x_{i+2})\bmod n,
    \end{align*}
    and 
    \begin{align*}
        \varphi(x_p) &\equiv \varphi(x_2\tr x_{p+1})\\
        &\equiv \varphi(x_2)\tr \varphi(x_{p+1})\\
        &\equiv 2\varphi(x_{p+1})- \varphi(x_2)\bmod n.
    \end{align*}
    Combining these,
    \begin{align*}
        \varphi(x_2)&\equiv 2\varphi(x_3)-\varphi(x_4)\\
        &\equiv 2(2(\varphi(x_4)-\varphi(x_5))-\varphi(x_4)\\
        &\equiv 3\varphi(x_4)-2\varphi(x_5)\\
        &\phantom{000000000}\vdots\\
        &\equiv (p-1)\varphi(x_p)-(p-2)\varphi(x_{p+1})\\
        &\equiv (p-1)(2\varphi(x_{p+1})-\varphi(x_2))-(p-2)\varphi(x_{p+1})\\
        &\equiv p\varphi(x_{p+1})-(p-1)\varphi(x_2)\bmod n,
    \end{align*}
    which yields $p\varphi(x_2)\equiv p\varphi(x_{p+1})\bmod n$.
    Similarly, 
    \begin{align*}
        \varphi(x_1)&\equiv 2\varphi(x_2)-\varphi(x_3)\\
        &\equiv 2(2(\varphi(x_3)-\varphi(x_4))-\varphi(x_3)\\
        &\equiv 3\varphi(x_3)-2\varphi(x_4)\\
        &\phantom{0000000000}\vdots\\
        &\equiv p\varphi(x_p)-(p-1)\varphi(x_{p+1})\\
        &\equiv p(2\varphi(x_{p+1})-\varphi(x_2))-(p-1)\varphi(x_{p+1})\\
        &\equiv 2p\varphi(x_{p+1})-p\varphi(x_2)-(p-1)\varphi(x_{p+1})\\
        &\equiv 2p\varphi(x_{p+1})-p\varphi(x_{p+1})-(p-1)\varphi(x_{p+1})\\
        &\equiv \varphi(x_{p+1})\bmod n.
    \end{align*}
\end{proof}

\begin{corollary}\label{cor:quandlesconsidered}
    Let $p,n\in\mathbb{N}$. If $y_1\not\equiv y_2\bmod n$, then $\hom(P(\widetilde{\mathcal T(p,2)}), (\mathbb Z_n,y_1,y_2))= \emptyset$.
\end{corollary}
\begin{proof}
    We give a proof by contrapositive. Suppose $\varphi\in \hom(P(\widetilde{\mathcal T(p,2)}), (\mathbb Z_n,y_1,y_2))$. By Remark~\ref{rem:Q(T(p,2)}, the basepoints of $P(\widetilde{\mathcal T(p,2)})$ are $x_1$ and $x_{p+1}$, so since $\varphi$ is a pointed quandle homomorphism, $\varphi(x_1)\equiv y_1\bmod n$ and $\varphi(x_{p+1})\equiv y_2\bmod n$. By Lemma~\ref{lem:Q(T(p,2))homprops}, $\varphi(x_1)\equiv \varphi(x_{p+1})\bmod n$. Thus,
    $y_1\equiv y_2\bmod n$.
\end{proof}

Corollary~\ref{cor:quandlesconsidered} gives us all we need to compute $\mathcal{Q}_{(\mathbb{Z}_n,y_1,y_2)}(\widetilde{\mathcal{T}(p,2)})$ when $y_1\not\equiv y_2\bmod n$.

\begin{corollary}
    Let $p,n\in\mathbb{N}$. Let $y_1,y_2\in\mathbb{Z}_n$. If $y_1\not\equiv y_2\bmod n$, then $\mathcal{Q}_{(\mathbb{Z}_n,y_1,y_2)}(\widetilde{\mathcal{T}(p,2)})=(\emptyset,\emptyset)$. 
\end{corollary}
\begin{proof}
    By Corollary~\ref{cor:quandlesconsidered}, $V(\mathcal{Q}_{(\mathbb{Z}_n,y_1,y_2)}(\widetilde{\mathcal{T}(p,2)}))=\emptyset$, so $\mathcal{Q}_{(\mathbb{Z}_n,y_1,y_2)}(\widetilde{\mathcal{T}(p,2)})$ is a quiver with no vertices and no edges.
\end{proof}

With the trivial case resolved, in what follows, we will only consider unpointed dihedral quandles $\mathbb{Z}_n$ and $2$-pointed dihedral quandles with one repeated basepoint $\mathcal{Z}$.
The following lemmas and theorem describe the coloring set of $\widetilde{\mathcal T(p,2)}$ with respect to $\mathcal Z$. 

\begin{lemma}\label{lem:determininghoms}
    Let $p,n\in\mathbb{N}$ and $\alpha, \beta\in \hom(Q(\widetilde{\mathcal T(p,2)}), \mathbb Z_n)$. If $\alpha(x_2)\equiv \beta(x_2)\bmod n$ and $\alpha(x_{p+1})\equiv \beta(x_{p+1})\bmod n$, then $\alpha=\beta$.
\end{lemma}
\begin{proof}
    Assume $\alpha(x_2)\equiv \beta(x_2)\bmod n$ and $\alpha(x_{p+1})\equiv \beta(x_{p+1})\bmod n$. Since $\alpha$ and $\beta$ are quandle homomorphisms, it suffices to show that they agree on the generators of $Q(\widetilde{\mathcal T(p,2)})$. By assumption, $\alpha$ and $\beta$ agree on $x_2$ and $x_{p+1}$. Thus, since $x_p=x_2\tr x_{p+1}$ by Remark~\ref{rem:Q(T(p,2)}, 
    $$\alpha(x_p)\equiv\alpha(x_2)\tr\alpha(x_{p+1})\equiv\beta(x_2)\tr\beta(x_{p+1})\equiv\beta(x_p)\bmod n.$$    
    Now we use strong induction on the rest of the generators, counting down from $p-1$, our base case. By Remark~\ref{rem:Q(T(p,2)}, $x_{p-1}=x_{p+1}\tr x_{p}$, so
    $$\alpha(x_{p-1})\equiv\alpha(x_{p+1})\tr\alpha(x_{p})\equiv\beta(x_{p+1})\tr\beta(x_{p})\equiv\beta(x_{p-1})\bmod n.$$
    Now assume that for some $i\in\{1,\dots, p-2\}$, for all $j>i$, $\alpha(x_j)\equiv\beta(x_j)\bmod n$. Then, in particular, $\alpha(x_{i+1})\equiv\beta(x_{i+1})\bmod n$ and $\alpha(x_{i+2})\equiv\beta(x_{i+2})\bmod n$. Again, by Remark~\ref{rem:Q(T(p,2)}, $x_i=x_{i+2}\tr x_{i+1}$, so 
    $$\alpha(x_{i})\equiv\alpha(x_{i+2})\tr\alpha(x_{i+1})\equiv\beta(x_{i+2})\tr\beta(x_{i+1})\equiv\beta(x_{i})\bmod n.$$
    Thus, by strong induction, $\alpha(x_i)\equiv\beta(x_i)\bmod n$ for all $i\in\{1,2,\dots,p+1\}$, so $\alpha=\beta$.
\end{proof}

Lemma~\ref{lem:determininghoms} establishes that colorings of the $\widetilde{\mathcal T(p,2)}$ linkoids with respect to the unpointed dihedral quandle are entirely determined by the images of $x_2$ and $x_{p+1}$. The following lemma establishes the form of many of the colorings of $\widetilde{\mathcal T(p,2)}$ with respect to $\mathcal Z$.

\begin{lemma}\label{lem:homsformula}
    Let $p,n\in\mathbb{N}$, $\gcd(p,n)=c$ and $d=\tfrac{n}{c}$. For each $x\in\mathbb Z_n$, let $L(x)$ be the least nonnegative residue of $x$ modulo $p$. Suppose $x_i$ is a generator of $P(\widetilde{\mathcal T(p,2)})$. For each $k\in\{0,1,\dots,c-1\}$, define $\alpha_k:P(\widetilde{\mathcal T(p,2)})\to \mathcal Z$ by 
    $$\alpha_k(x_i)\equiv y+L(i-1)\cdot kd\bmod n.$$
    Let 
    $$\mathcal A=\{\alpha_k\mid k\in\{0,1,\dots,c-1\}\}.$$
    Then $\mathcal A\subseteq \hom(P(\widetilde{\mathcal T(p,2)}), \mathcal Z)$.
\end{lemma}
\begin{proof}
    Let $\alpha_k\in\mathcal A$. Since $\alpha_k$ is defined on the generators of $P(\widetilde{\mathcal T(p,2)})$, we can extend $\alpha_k$ linearly to the rest of $P(\widetilde{\mathcal T(p,2)})$. 
    Further, to show that $\alpha_k$ is a pointed quandle homomorphism, it suffices to show that $\alpha_k$ preserves the basepoints and relations of $P(\widetilde{\mathcal T(p,2)})$. 
    For the basepoints, see that 
    $$\alpha_k(x_1)\equiv y+(1-1)kd\equiv y\bmod n$$
    and 
    $$\alpha_k(x_{p+1})\equiv y+L(p+1-1)\cdot kd\equiv y+L(p)kd\equiv y\bmod n.$$
    Since $c=\gcd(p,n)$, for some $s\in\mathbb Z$, $p=sc$, so $$(p-1)kd\equiv -kd+pkd\equiv -kd+sck\tfrac{n}{c}\equiv -kd+skn\equiv -kd\bmod n.$$
    Thus, for the first relation, $x_p=x_2\tr x_{p+1}$, observe that 
    \begin{align*}
        \alpha_k(x_p)&\equiv y+(p-1)kd\\
        &\equiv y-kd\\
        &\equiv 2y-(y+kd) \\
        &\equiv (y+kd)\tr y\\
        &\equiv \alpha_k(x_2)\tr \alpha_k(x_{p+1})\bmod n,
    \end{align*}
    as desired. The remaining relations are $x_i=x_{i+2}\tr x_{i+1}$ for each $i\in\{1,2,\dots,p-1\}$.
    Well,
    \begin{align*}
        \alpha_k(x_i) &\equiv y+(i-1)kd\\
        &\equiv 2(y+ikd)-(y+(i+1)kd)\\
        &\equiv (y+(i+1)kd)\tr (y+ikd)\\
        &\equiv \alpha_k(x_{i+2})\tr \alpha_k(x_{i+1})\bmod n,
    \end{align*}
    as desired. Thus, $\alpha_k\in \hom(P(\widetilde{\mathcal T(p,2)}),\mathcal Z)$, so $\mathcal A\subseteq \hom(P(\widetilde{\mathcal T(p,2)}),\mathcal Z)$.
\end{proof}

It turns out that $\mathcal A$ is the entire coloring set of $\widetilde{T(p,2)}$ with respect to $\mathcal Z$.

\begin{theorem}\label{thm:hom}

    Let $p,n\in\mathbb{N}$. Let $c=\gcd(p,n)$ and $d=\tfrac{n}{c}$. Let $\mathcal A$ be defined as in Lemma~\ref{lem:homsformula}. Then $\hom(P(\widetilde{\mathcal T(p,2)}),\mathcal Z)=\mathcal A$ and thus 
    $$\Phi_{\mathcal Z}^{\mathbb Z}(\widetilde{\mathcal T(p,2)})=|\hom(P(\widetilde{\mathcal T(p,2)}), \mathcal Z)|=c.$$

\end{theorem}
\begin{proof}

    By Lemma~\ref{lem:homsformula}, $\mathcal A\subseteq \hom(P(\widetilde{\mathcal T(p,2)}),\mathcal Z)$. It remains to be shown that $\hom(P(\widetilde{\mathcal T(p,2)}),\mathcal Z)\subseteq \mathcal A$. 
    Let $\alpha\in \hom(P(\widetilde{\mathcal T(p,2)}), \mathcal Z)$. 
    Since pointed quandle homomorphisms preserve basepoints, $\alpha(x_1)\equiv\alpha(x_{p+1})\equiv y$. 
    Further, by Lemma~\ref{lem:Q(T(p,2))homprops}, $p\alpha(x_2)\equiv p\alpha(x_{p+1})\bmod n$, so $p\alpha(x_2)\equiv py\bmod n$. 
    Equivalently,  $\alpha(x_2)\equiv y\bmod \tfrac{n}{c}$. Hence, for some $k\in\{0,1,\dots,c-1\}$, $$\alpha(x_2)\equiv y+k\tfrac{n}{c}\equiv y+kd\bmod n.$$ 
    Thus, by Lemma~\ref{lem:determininghoms}, $\alpha=\alpha_k\in\mathcal A$, so $\hom(P(\widetilde{\mathcal T(p,2)}), \mathcal Z)\subseteq \mathcal A$. Thus, $\hom(P(\widetilde{\mathcal T(p,2)}), \mathcal Z)=\mathcal A$. 

    Further, it is evident that $|\mathcal A|\le c$.
    For $i,j\in\{0,1,\dots,c-1\}$ with $i\ne j$, $$\alpha_i(x_2)\equiv y+id\not\equiv y+jd\equiv \alpha_j(x_2)\bmod n,$$ so $\alpha_i\ne \alpha_j$. Hence, $|\mathcal A|\ge c$, so $|\mathcal A|=c$ and thus
    $$\Phi_{\mathcal Z}^{\mathbb Z}(\widetilde{\mathcal T(p,2)})=|\hom(P(\widetilde{\mathcal T(p,2)}),\mathcal Z)|=|\mathcal A|=c.$$
\end{proof}

Now we are ready to compute the quandle counting matrix of $\widetilde{\mathcal T(p,2)}$ with respect to $\mathbb Z_n$.

\begin{corollary}\label{cor:ZnUnpointedEnds}
    Let $p,n\in\mathbb{N}$. Let $c=\gcd(p,n)$. Then $\Phi_{\mathbb Z_n}^{M_n}(\widetilde{\mathcal T(p,2)})=cI_n$.
\end{corollary}
\begin{proof}
    Let $(i,j)\in \mathbb Z_n\times \mathbb Z_n$. If $i=j$, then by Theorem~\ref{thm:hom}, $$\left(\Phi_{\mathbb Z_n}^{M_n}(\widetilde{\mathcal T(p,2)})\right)_{i,j}=\left(\Phi_{\mathbb Z_n}^{M_n}(\widetilde{\mathcal T(p,2)})\right)_{i,i}=\Phi_{(\mathbb Z_n,i,i)}^{\mathbb Z}(P(\widetilde{\mathcal T(p,2)})=|\hom(P(\widetilde{\mathcal T(p,2)}),(\mathbb Z_n,i,i))|=c,$$
    and if instead $i\ne j$, then by Corollary~\ref{cor:quandlesconsidered}, 
    $$\left(\Phi_{\mathbb Z_n}^{M_n}(\widetilde{\mathcal T(p,2)})\right)_{i,j}=\Phi_{(\mathbb Z_n,i,j)}^{\mathbb Z}(P(\widetilde{\mathcal T(p,2)})=|\hom(P(\widetilde{\mathcal T(p,2)}),(\mathbb Z_n,i,j))|=0.$$
\end{proof}

In $\mathcal Q_{\mathcal Z}(\widetilde{\mathcal T(p,2)})$, there is an arc from $\alpha\in V=\hom(P(\widetilde{\mathcal T(p,2)}),\mathcal Z)$ to $\beta\in V$ for each $\varphi\in \End(\mathcal Z)$ such that $\varphi\circ \alpha=\beta$, so we need some understanding of $\End(\mathcal Z)$.

\begin{theorem}\label{thm:dihedralends}
    Let $n\in\mathbb{N}$.
    For each $k\in \mathbb Z_n$, define $\varphi_k:\mathbb Z_n \to \mathbb Z_n$ by
    \begin{align*}
        \varphi_k(i)\equiv \begin{cases}
            y\bmod n &\text{ if } i\equiv y\bmod n,\\
            k\bmod n &\text{ if } i\equiv y+1\bmod n, \text{ and}\\
            \varphi_k(i-2)\tr\varphi_k(i-1)\bmod n &\text{ otherwise.}
        \end{cases}
    \end{align*}
    Then
    $\End(\mathcal Z)=\{\varphi_k\mid k\in \mathbb Z_n\}$ and hence $|\End(\mathcal Z)|=n$.
\end{theorem}
\begin{proof}

    First, let $x_1,x_2\in\mathbb Z_n$. By definition of modular addition, for some $i\in \mathbb Z_n$, $x_2\equiv x_1+i\bmod n$. Further, by definition of the dihedral quandle, for any $x\in\mathbb Z_n$, $x\tr (x+1)\equiv 2(x+1)-x\equiv x+2\bmod n$.
    Since $x_2\equiv x_1+i\bmod n$, from induction on the previous fact, it follows that $x_2$ is in the subquandle of $\mathbb Z_n$ generated by $\{x_1,x_1+1\}$. That is, any two consecutive elements of $\mathbb Z_n$ generate $\mathbb Z_n$.
    
    In particular, that means $\{y,y+1\}$ generates $\mathbb Z_n$, and any function $\varphi:\{y,y+1\}\to \mathbb Z_n$ extended linearly yields a quandle endomorphism of $\mathbb Z_n$. Since there are $n$ choices for both $\varphi(y)$ and $\varphi(y+1)$, $|\End(\mathbb Z_n)|\ge n^2$. Further, if $\varphi_1,\varphi_2\in \End(\mathbb Z_n)$ and $\varphi_1(y)\equiv \varphi_2(y)\bmod n$ and $\varphi_1(y+1)\equiv\varphi_2(y+1)\bmod n$, then, since $\varphi_1$ and $\varphi_2$ are homomorphisms and $\{y,y+1\}$ generates $\mathbb Z_n$, $\varphi_1=\varphi_2$.

    Thus, any $\varphi\in\End(\mathbb Z_n)$ is determined entirely by $\varphi(y)$ and $\varphi(y+1)$ and hence $|\End(\mathbb Z_n)|\le n^2$, so $|\End(\mathbb Z_n)|=n^2$.

    By definition, $\End(\mathcal Z)=\{\varphi\in \End(\mathbb Z_n)\mid \varphi(y)=y\}$. 
    Thus, each pointed endomorphism of $\mathcal Z$ is determined by its image of $y+1$ alone, and hence, for each $\varphi\in \End(\mathcal Z)$, $\varphi=\varphi_k$, where $k\equiv\varphi(y+1)\bmod n$ and $\varphi_k$ is defined as above. Moreover, for each $k\in\mathbb Z_n$, $\varphi_k\in \End(\mathcal Z)$ since each $\varphi_k(y)\equiv y\bmod n$. Thus, $\End(\mathcal Z)=\{\varphi_k\mid k\in\mathbb Z_n\}$.
\end{proof}

\begin{corollary}\label{cor:dihedralendsvalues}
    Let $n\in\mathbb{N}$.
    Then for each $\varphi_k\in \End(\mathcal Z)$ and $i\in\mathbb Z_n$, $\varphi_k(y+i)\equiv ik-(i-1)y \bmod n$.
\end{corollary}
\begin{proof}
    We use strong induction. Our base case is $i=1$. By definition of $\varphi_k$, $\varphi_k(y+1)\equiv k\equiv 1 k-(1-1)y\bmod n$.

    Now assume that for some $j\in \mathbb Z_n$, for all $i\in\{y,y+1,\dots, y+j\}$, $\varphi_k(y+i)\equiv ik-(i-1)y\bmod n$. Then in particular, $\varphi_k(y+(j-1))\equiv (j-1)k-((j-1)-1)y\bmod n$ and $\varphi_k(y+j)\equiv jk-(j-1)y\bmod n$. Then by definition of $\varphi_k$, 
    \begin{align*}
        \varphi_k(y+(j+1)) &\equiv  \varphi_k(y+(j-1))\tr\varphi_k(y+j)\\
        &\equiv((j-1)k-((j-1)-1)y)\tr ( jk-(j-1)y)\\
        &\equiv 2( jk-(j-1)y)-((j-1)k-((j-1)-1)y)\\
        &\equiv (j+1)k - ((j+1)-1)y\bmod n.
    \end{align*}
\end{proof}

To coherently describe $\mathcal Q_{\mathcal Z}(\widetilde{\mathcal T(p,2)})$, we need some notation to denote certain types of directed multigraphs.
Recall from Definition~\ref{def:multigraph} that a directed multigraph is an ordered pair $G=(V, w)$, where $V$ is a finite set and $w:V\times V\to \mathbb N\cup\{0\}$ is a function. 

\begin{definition}\label{def:Knk}
    We say a directed multigraph $G$ is a \emph{complete $k$-regular directed multigraph on $n$ vertices} if $G\cong K_{n,k}=(\{1,2,\dots n\}, w)$, where, for all $u,v\in \{1,2,\dots,n\}$, $w(u,v)=k$. 
\end{definition}

That is, a directed multigraph $G$ is a complete $k$-regular directed multigraph on $n$ vertices if $|V(G)|=n$ and there are $k$ arcs from $u$ to $v$ for all $u,v\in V(G)$.

\begin{definition}
    Let $G=(V(G),w_G)$ and $H=(V(H),w_H)$ be two directed multigraphs such that $V(G)\cap V(H)=\emptyset$. Then the \emph{$n$-directed join of $G$ to $H$} is the directed multigraph $G \overrightarrow{\triangledown_n} H=(V, w)$, where $V=V(G)\cup V(H)$ and $w(u,v):(V(G)\cup V(H))\times (V(G)\cup V(H))\to \mathbb N\cup \{0\}$ is defined by 
    \begin{align*}
        w(u,v) = \begin{cases}
            w_G(u,v) &\text{ if } u,v\in V(G)\\
            w_H(u,v) &\text{ if } u,v\in V(H)\\
            n &\text{ if } u\in V(G) \text{ and } v\in V(H), \text{ and }\\
            0 &\text{ if } u\in V(H) \text{ and } v\in V(G).
        \end{cases}
    \end{align*}
    Note that the subgraphs of $G\overrightarrow{\triangledown_n}H$ induced by $V(G)$ and $V(H)$ are $G$ and $H$, respectively.
\end{definition}

\begin{example}\label{ex:graphexample}
    Let $G\cong K_{2,1}$ and $H\cong K_{4, 1}$ such that $V(G)\cap V(H)=\emptyset$. Then $G\overrightarrow{\triangledown_1} H$ has six vertices, two in $V(G)$ and four in $V(H)$. In Figure~\ref{fig:graphexample}, the two vertices on the left are those of $V(G)$ and the four arranged in a square on the right are those of $V(H)$.

    \begin{figure}
        \centering
        \includegraphics[width=0.5\linewidth]{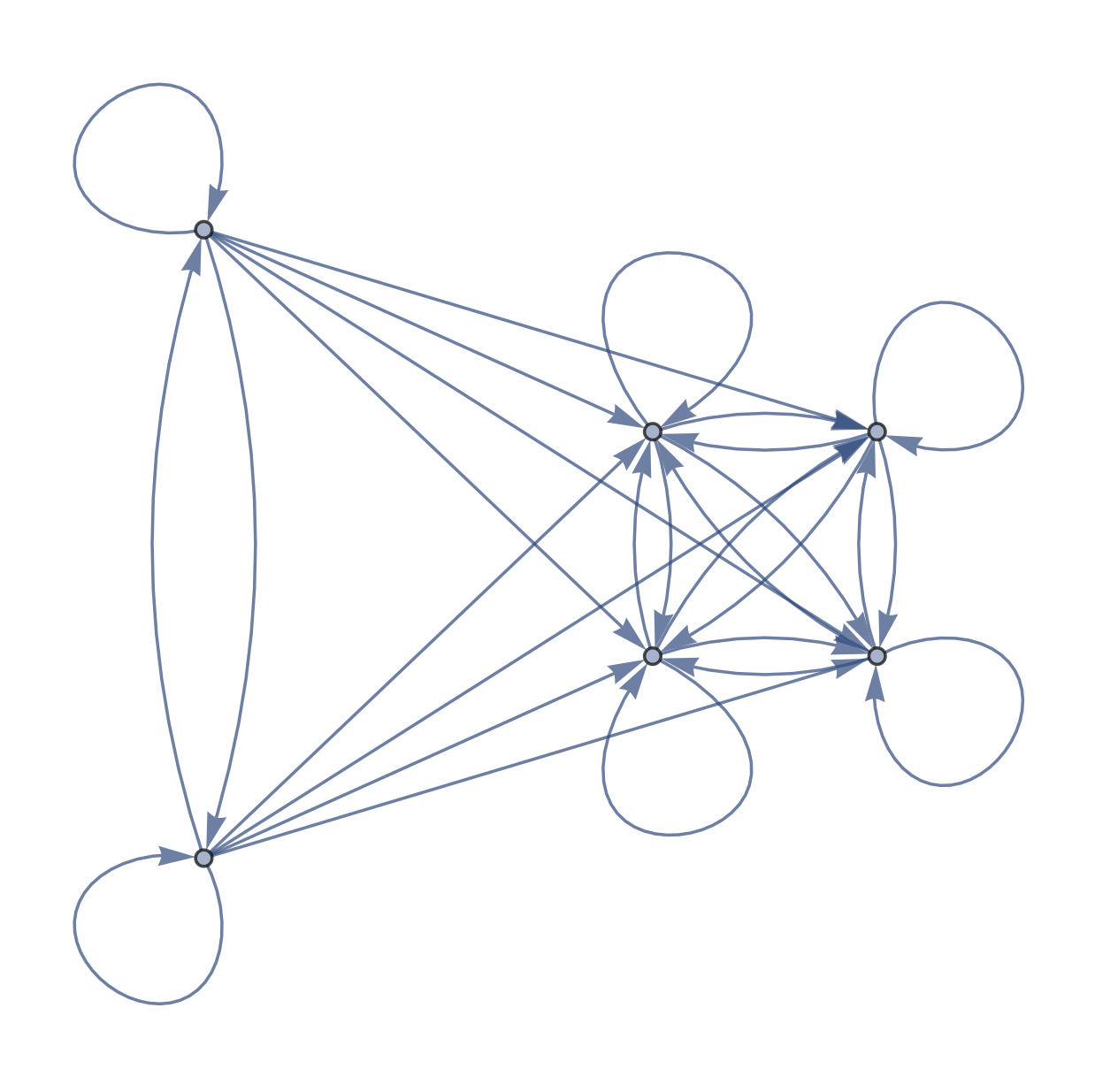}
        \caption{The directed multigraph $G\overrightarrow{\triangledown_1} H$ from Example~\ref{ex:graphexample}.}
        \label{fig:graphexample}
    \end{figure}
\end{example}

\begin{theorem}\label{thm:quivergcd1}
    Let $p,n\in\mathbb{N}$.
    If $\gcd(p,n)=1$, then $\mathcal Q_{\mathcal Z}(\widetilde{\mathcal T(p,2)})\cong K_{1,n}$.
\end{theorem}
\begin{proof}
    Assume $\gcd(p,n)=1$. Then by Theorem~\ref{thm:hom}, $\hom(P(\widetilde{\mathcal T(p,2)}),\mathcal Z)=\{\alpha_0\}$, where $\alpha_0$ is the trivial coloring. Hence, by definition, $V(\mathcal Q_{\mathcal Z}(\widetilde{\mathcal T(p,2)}))=\{\alpha_0\}$, and thus any arc in $\mathcal Q_{\mathcal Z}(\widetilde{\mathcal T(p,2)})$ is a loop at $\alpha_0$.

    Let $\varphi_k\in\End(\mathcal Z)$. By Lemma~\ref{lem:determininghoms}, $\varphi_k\circ\alpha_0$ is completely determined by $\varphi_k\circ\alpha_0(x_2)$ and $\varphi_k\circ\alpha_0(x_{p+1})$.
    Since $\alpha_0$ is the trivial coloring, $\alpha_0(x_2)\equiv\alpha_0(x_{p+1})\equiv y\bmod n$, and
    by Theorem~\ref{thm:dihedralends}, $\varphi_k(y)\equiv y\bmod n$. Hence, $\varphi_k\circ\alpha_0(x_2)\equiv \varphi_k(y)\equiv y\equiv \alpha_0(x_2)\bmod n$ and $\varphi_k\circ\alpha_0(x_{p+1})\equiv \varphi_k(y)\equiv y\equiv\alpha_0(x_{p+1})\bmod n$, so by Lemma~\ref{lem:determininghoms}, $\varphi_k\circ \alpha_0=\alpha_0$. 
    
    Hence, for all $\varphi_k\in\End(\mathcal Z)$, $\varphi_k\circ\alpha_0=\alpha_0$, so there are $n$ loops at $\alpha_0$ in $\mathcal Q_{\mathcal Z}(\widetilde{\mathcal T(p,2)})$. Thus, $\mathcal Q_{\mathcal Z}(\widetilde{\mathcal T(p,2)})\cong K_{1,n}$.
\end{proof}

\begin{lemma}\label{lem:composingendhoms}
    Let $p,n\in\mathbb{N}$.
    If  $c=\gcd(p,n)$ is prime, then for all $\varphi_i,\varphi_j\in \End(\mathcal Z)$ and $\alpha_k\in\hom(P(\widetilde{\mathcal T(p,2)}),\mathcal Z)$ such that $\alpha_k$ is nontrivial, $\varphi_i\circ \alpha_k=\varphi_j\circ\alpha_k$ if and only if $i\equiv j\bmod c$.
\end{lemma}
\begin{proof}
    Let $\varphi_i,\varphi_j\in \End(\mathcal Z)$ and let $\alpha_k\in \hom(P(\widetilde{\mathcal T(p,2)}),\mathcal Z)$ be nontrivial. Since $\alpha_k$ is nontrivial, $k\in\{1,\dots,c-1\}$.

    Suppose $\varphi_i\circ\alpha_k=\varphi_j\circ\alpha_k$. 
    Then by Theorem~\ref{thm:hom}, for some $l\in \{0,1,\dots,c-1\}$, $\varphi_i\circ\alpha_k=\varphi_j\circ\alpha_k=\alpha_l$,
    $\alpha_k(x_2)\equiv y+kd\bmod n$, and $\alpha_l(x_2)\equiv y+ld\bmod n$, where $d=\frac{n}{c}$. Moreover, since $\varphi_i\circ\alpha_k=\alpha_l$ and $\varphi_j\circ\alpha_k=\alpha_l$, we have
    $$\varphi_i\circ\alpha_k(x_2)\equiv \alpha_l(x_2)\bmod n$$ and $$\varphi_j\circ\alpha_k(x_2)\equiv \alpha_l(x_2)\bmod n.$$
    Thus, $\varphi_i(y+kd)\equiv y+ld\bmod n$ and $\varphi_j(y+kd)\equiv y+ld\bmod n$. On the other hand, by Corollary~\ref{cor:dihedralendsvalues}, $\varphi_i(y+kd)\equiv kdi-(kd-1)y\bmod n$ and $\varphi_j(y+kd)\equiv kdj-(kd-1)y\bmod n$. 
    Hence, 
    $$kdi-(kd-1)y\equiv y+ld\bmod n \text{ and } kdj-(kd-1)y\equiv y+ld\bmod n.$$
    Distributing,
    $$kdi-kdy+y\equiv y+ld\bmod n \text{ and } kdj-kdy+y\equiv y+ld\bmod n,$$
    so
    $$kdi-kdy\equiv ld\bmod n \text{ and } kdj-kdy\equiv ld\bmod n.$$
    Equivalently, since $\tfrac{n}{d}=c$, 
    $$k(i-y)\equiv l\bmod c \text{ and } k(j-y)\equiv l\bmod c.$$
    Since $c$ is prime and $1\le k\le c-1$, $\gcd(k,c)=1$, so the equivalence $k(x-y)\equiv l\bmod c$ has a unique solution. Thus, $i\equiv j\bmod c$.

    Conversely, assume $i\equiv j\bmod c$. Then 
    $$k(i-y)\equiv k(j-y)\bmod c.$$ 
    Equivalently, since $c=\frac{n}{d}$, 
    $$kd(i-y)\equiv kd(j-y)\bmod n.$$ 
    Distributing and adding $y$ to both sides, we see
    $$kdi-kdy+y\equiv kdj-kdj-kdy+y\bmod n.$$ 
    Factoring,
    $$kdi-(kd-1)y\equiv kdj-(kd-1)y\bmod n.$$
    By Corollary~\ref{cor:dihedralendsvalues}, $\varphi_i(y+kd)\equiv kdi-(kd-1)y\bmod n$ and $\varphi_j(y+kd)\equiv kdj-(kd-1)y\bmod n$. Thus, since $\alpha_k(x_2)\equiv y+kd$, we have $$\varphi_i\circ\alpha_k(x_2)\equiv\varphi_j\circ\alpha_k(x_2)\bmod n.$$
     Additionally, since $\alpha_k$ is a valid $\mathcal{Z}$-coloring of $\widetilde{\mathcal T(p,2)}$, we have $\alpha_k(x_{p+1}) = y$. Thus, 
    $$\varphi_i\circ\alpha_k(x_{p+1})\equiv\varphi_j\circ\alpha_k(x_{p+1}) \equiv y\bmod n.$$
    Hence, by Lemma~\ref{lem:determininghoms}, $\varphi_i\circ\alpha_k=\varphi_j\circ\alpha_k$.
\end{proof}

\begin{theorem}\label{thm:quivergcdprime}
    Let $p,n\in\mathbb{N}$ with $\gcd(p,n)=c$, and $d=\tfrac{n}{c}$. 
    Define $G\cong K_{c-1,d}$ and $H\cong K_{1,n}$ such that $V(G)\cap V(H)=\emptyset$. 
    If $c$ is prime, then $\mathcal Q_{\mathcal Z}(\widetilde{\mathcal T(p,2)})\cong G\overrightarrow{\triangledown_d} H$.
\end{theorem}

\begin{proof}
    By definition of the full pointed quandle coloring quiver and by Theorem~\ref{thm:hom}, $$V(\mathcal Q_{\mathcal Z}(\widetilde{\mathcal T(p,2)}))=\hom(P(\widetilde{\mathcal T(p,2)}),\mathcal Z)=\{\alpha_i\in \hom(P(\widetilde{\mathcal T(p,2)}),\mathcal Z)\mid i\in\{0,1,\dots,c-1\}\},$$
    where $\alpha_0$ is the trivial coloring and, for all $i\in\{1,\dots,c-1\}$, $\alpha_i$ is nontrivial. As was shown in the proof of Theorem~\ref{thm:quivergcd1}, for all $\varphi_k\in\End(\mathcal Z)$, $\varphi_k\circ\alpha_0=\alpha_0$, so since $|\End(\mathcal Z)|=n$ by Theorem~\ref{thm:dihedralends}, the subgraph of $\mathcal Q_{\mathcal Z}(\widetilde{\mathcal T(p,2)})$ induced by $\{\alpha_0\}$ is isomorphic to $K_{1,n}$.

    Now let $\alpha_k\in V(\mathcal Q_{\mathcal Z}(\widetilde{\mathcal T(p,2)}))=\hom(P(\widetilde{\mathcal T(p,2)}),\mathcal Z)$ such that $\alpha_k\ne \alpha_0$. 
    Since $|\End(\mathcal Z)|=n$, in $\mathcal Q_{\mathcal Z}(\widetilde{\mathcal T(p,2)})$ there are $n$ arcs originating at $\alpha_k$. 
    Suppose $\varphi_i\in\End(\mathcal Z)$ and $\alpha_l\in\hom(P(\widetilde{\mathcal T(p,2)},\mathcal Z)$ such that $\varphi_i\circ\alpha_k=\alpha_l$. Then by Lemma~\ref{lem:composingendhoms}, for all $\varphi_j\in \End(\mathcal Z)$, $\varphi_j\circ \alpha_k=\alpha_l$ if and only if
    $\varphi_j\in \{\varphi_j\in \End(\mathcal Z)\mid j\equiv i\bmod c\}$. 
    Since $n=cd$ and $c$ is prime, for any fixed $i\in\mathbb Z_n$, there are $d$ distinct solutions in $\mathbb Z_n$ to $j\equiv i\bmod c$, so $|\{\varphi_j\in \End(\mathcal Z)\mid j\equiv i\bmod c\}|=d$.
    Now suppose instead that there exists $\alpha_l\in \hom(P(\widetilde{\mathcal T(p,2)},\mathcal Z)$ such that there does not exist $\varphi_i\in\End(\mathcal Z)$ such that $\varphi_i\circ\alpha_k=\alpha_l$, so in $\mathcal Q_{\mathcal Z}(\widetilde{\mathcal T(p,2)})$, $\alpha_l$ has no incoming arcs originating at $\alpha_k$. Then since there are $c-1$ other homomorphisms in $\hom(P(\widetilde{\mathcal T(p,2)},\mathcal Z)$, each of which with at most $d$ incoming arcs originating at $\alpha_k$, the total number of arcs originating at $\alpha_k$ is at most $(c-1)d<cd=n$, which contradicts there being $n$ arcs originating at $\alpha_k$. 
    Thus, for all $\alpha_i\in V(\mathcal Q_{\mathcal Z}(\widetilde{\mathcal T(p,2)}))= \hom(P(\widetilde{\mathcal T(p,2)},\mathcal Z)$, there are $d$ arcs from $\alpha_k$ to $\alpha_l$. 

    This means that the subgraph induced by the $c-1$ nontrivial homomorphisms is isomorphic to $K_{c-1,d}$, and for each nontrivial $\alpha_k\in\hom(P(\widetilde{\mathcal T(p,2)},\mathcal Z)$, there are $d$ arcs from $\alpha_k$ to $\alpha_0$. Thus, $\mathcal Q_{\mathcal Z}(\widetilde{\mathcal T(p,2)})\cong G\overrightarrow{\triangledown_d} H$.
\end{proof}

\begin{corollary}\label{cor:torusIndegreeMatrices}
    Let $p,n\in\mathbb{N}$ with $\gcd(p,n)=c$ and $d=\tfrac{n}{c}$. Then 
    \begin{enumerate}
        \item if $c=1$, $\Phi_{\mathbb{Z}_n}^{M_n,\textup{deg+}}(\widetilde{\mathcal{T}(p,2)})=(u^c)I_n$, and
        \item if $c$ is prime, $\Phi_{\mathbb{Z}_n}^{M_n,\textup{deg+}}(\widetilde{\mathcal{T}(p,2)})=(u^{n+(c-1)d}+(c-1)u^{(c-1)d})I_n$.
    \end{enumerate}
\end{corollary}

\begin{example}
    We consider the $\widetilde{\mathcal T(10,2)}$ torus linkoid and some of its invariants with respect to $\mathbb{Z}_5$. 
    Note that $c=\gcd(10,5)=5$ and $d=\frac{5}{5}=1$.
    Hence, by Corollary~\ref{cor:ZnUnpointedEnds}, \[\Phi_{\mathbb{Z}_5}^{M_5}(\widetilde{\mathcal{T}(10,2)})=
    \begin{bmatrix}
        5 & 0 & 0 & 0 & 0 \\
        0 & 5 & 0 & 0 & 0 \\
        0 & 0 & 5 & 0 & 0 \\
        0 & 0 & 0 & 5 & 0 \\
        0 & 0 & 0 & 0 & 5
    \end{bmatrix},\]
    and by Corollary~\ref{cor:torusIndegreeMatrices},
    \[
    \Phi_{\mathbb{Z}_5}^{M_5,\textup{deg+}}(\widetilde{\mathcal{T}(10,2)})
    = \begin{bmatrix}
        u^9+4u^4 & 0 & 0 & 0 & 0 \\
        0 & u^9+4u^4 & 0 & 0 & 0 \\
        0 & 0 & u^9+4u^4 & 0 & 0 \\
        0 & 0 & 0 & u^9+4u^4 & 0 \\
        0 & 0 & 0 & 0 & u^9+4u^4
    \end{bmatrix}.
    \]
    Now we fix $y=0\in \mathbb{Z}_n$. 
    By Theorem 6.6, we have $5$ $\mathcal{Z}$-colorings of $\widetilde{\mathcal{T}(10,2)}$:
    \[
    \begin{tabular}{l|ccccccccccc}
     & $x_1$ & $x_2$ & $x_3$ & $x_4$ & $x_5$ & $x_6$ & $x_7$ & $x_8$ & $x_9$ & $x_{10}$ & $x_{11}$ \\
    \hline
    $\alpha_0$ & $0$ & $0$ & $0$ & $0$ & $0$ & $0$ & $0$ & $0$ & $0$ & $0$ & $0$ \\
    $\alpha_1$ & $0$ & $1$ & $2$ & $3$ & $4$ & $0$ & $1$ & $2$ & $3$ & $4$ & $0$ \\
    $\alpha_2$ & $0$ & $2$ & $4$ & $1$ & $3$ & $0$ & $2$ & $4$ & $1$ & $3$ & $0$ \\
    $\alpha_3$ & $0$ & $3$ & $1$ & $4$ & $2$ & $0$ & $3$ & $1$ & $4$ & $2$ & $0$ \\
    $\alpha_4$ & $0$ & $4$ & $3$ & $2$ & $1$ & $0$ & $4$ & $3$ & $2$ & $1$ & $0$. 
    \end{tabular}
    \]
    By Theorem~\ref{thm:dihedralends}, there are $5$ endomorphisms of $\mathcal{Z}$:
    \[
    \begin{tabular}{l|ccccc}
      & $0$ & $1$ & $2$ & $3$ & $4$ \\
    \hline
    $\varphi_0$ & $0$ & $0$ & $0$ & $0$ & $0$ \\
    $\varphi_1$ & $0$ & $1$ & $2$ & $3$ & $4$ \\
    $\varphi_2$ & $0$ & $2$ & $4$ & $1$ & $3$ \\
    $\varphi_3$ & $0$ & $3$ & $1$ & $4$ & $2$ \\
    $\varphi_4$ & $0$ & $4$ & $3$ & $2$ & $1$.
\end{tabular}
\]
    Lastly, let $G\cong K_{4,1}$ and $H\cong K_{1,5}$ such that $V(G)\cap V(H)=\emptyset$. Then by Theorem~\ref{thm:quivergcdprime}, $\mathcal{Q}_{\mathcal{Z}}(\widetilde{\mathcal{T}(10,2)})\cong G\overrightarrow{\triangledown_1} H$. Figure~\ref{fig:quiverex} is a drawing of $\mathcal{Q}_{\mathcal{Z}}(\widetilde{\mathcal{T}(10,2)})$.

    \begin{figure}
        \centering
        \includegraphics[width=0.5\linewidth]{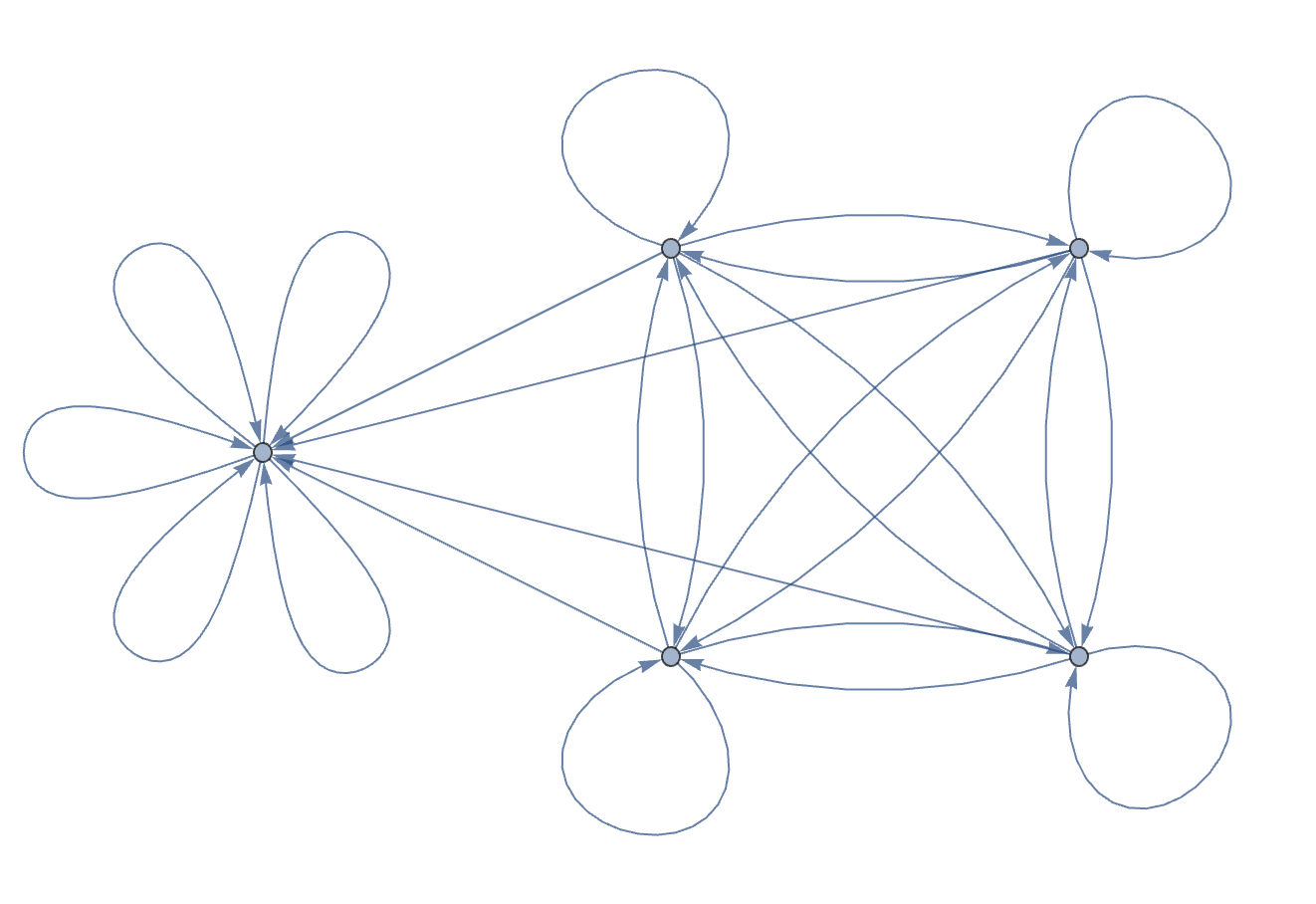}
        \caption{A drawing of $\mathcal Q_{\mathcal Z}(\widetilde{\mathcal T(10,2)})$.}
        \label{fig:quiverex}
    \end{figure}
\end{example}

\bibliography{Ref}

\begin{bibdiv}
\begin{biblist}

\bib{BG}{article}{
      author={Barbensi, Agnese},
      author={Goundaroulis, Dimos},
       title={f-distance of knotoids and protein structure},
        date={2021},
     journal={Proceedings of the Royal Society A},
      volume={477},
      number={2246},
       pages={20200898},
}

\bib{BaCa}{article}{
      author={Basi, Jagdeep},
      author={Caprau, Carmen},
       title={Quandle coloring quivers of {$(p,2)$}-torus links},
        date={2022},
        ISSN={0218-2165,1793-6527},
     journal={J. Knot Theory Ramifications},
      volume={31},
      number={9},
       pages={Paper No. 2250057, 14},
         url={https://doi.org/10.1142/S0218216522500572},
      review={\MR{4475496}},
}

\bib{BBA}{article}{
      author={Bataineh, Khaled},
      author={Batayneh, Fawwaz},
      author={Alkasasbeh, Ahmad~H.},
       title={A new invariant of planar knotoids and finite-type invariants},
        date={2023},
        ISSN={0218-2165,1793-6527},
     journal={J. Knot Theory Ramifications},
      volume={32},
      number={12},
       pages={Paper No. 2350077, 16},
         url={https://doi.org/10.1142/S0218216523500773},
      review={\MR{4688852}},
}

\bib{CJKLS}{article}{
      author={Carter, J.~Scott},
      author={Jelsovsky, Daniel},
      author={Kamada, Seiichi},
      author={Langford, Laurel},
      author={Saito, Masahico},
       title={Quandle cohomology and state-sum invariants of knotted curves and
  surfaces},
        date={2003},
        ISSN={0002-9947,1088-6850},
     journal={Trans. Amer. Math. Soc.},
      volume={355},
      number={10},
       pages={3947\ndash 3989},
         url={https://doi.org/10.1090/S0002-9947-03-03046-0},
      review={\MR{1990571}},
}

\bib{CCN}{article}{
      author={Ceniceros, Jose},
      author={Christiana, Anthony},
      author={Nelson, Sam},
       title={Psyquandle coloring quivers},
        date={2023},
        ISSN={0218-2165,1793-6527},
     journal={J. Knot Theory Ramifications},
      volume={32},
      number={11},
       pages={Paper No. 2350073, 18},
         url={https://doi.org/10.1142/S0218216523500736},
      review={\MR{4683265}},
}

\bib{CN}{article}{
      author={Cho, Karina},
      author={Nelson, Sam},
       title={Quandle coloring quivers},
        date={2019},
        ISSN={0218-2165,1793-6527},
     journal={J. Knot Theory Ramifications},
      volume={28},
      number={1},
       pages={1950001, 12},
         url={https://doi.org/10.1142/S0218216519500019},
      review={\MR{3910948}},
}

\bib{CN1}{article}{
      author={Cho, Karina},
      author={Nelson, Sam},
       title={Quandle coloring quivers},
        date={2019},
        ISSN={0218-2165,1793-6527},
     journal={J. Knot Theory Ramifications},
      volume={28},
      number={1},
       pages={1950001, 12},
         url={https://doi.org/10.1142/S0218216519500019},
      review={\MR{3910948}},
}

\bib{DDFS}{article}{
      author={Dorier, Julien},
      author={Goundaroulis, Dimos},
      author={Benedetti, Fabrizio},
      author={Stasiak, Andrzej},
       title={{Knoto-ID: a tool to study the entanglement of open protein
  chains using the concept of knotoids}},
        date={201805},
        ISSN={1367-4803},
     journal={Bioinformatics},
      volume={34},
      number={19},
       pages={3402\ndash 3404},
  eprint={https://academic.oup.com/bioinformatics/article-pdf/34/19/3402/48920011/bioinformatics\_34\_19\_3402.pdf},
         url={https://doi.org/10.1093/bioinformatics/bty365},
}

\bib{EJL}{misc}{
      author={Elhamdadi, Mohamed},
      author={Jones, Brooke},
      author={Liu, Minghui},
       title={Quandle coloring quivers of general torus links by dihedral
  quandles},
        date={2024},
         url={https://arxiv.org/abs/2403.04534},
}

\bib{EN}{book}{
      author={Elhamdadi, Mohamed},
      author={Nelson, Sam},
       title={Quandles---an introduction to the algebra of knots},
      series={Student Mathematical Library},
   publisher={American Mathematical Society, Providence, RI},
        date={2015},
      volume={74},
        ISBN={978-1-4704-2213-4},
         url={https://doi-org.ez.hamilton.edu/10.1090/stml/074},
      review={\MR{3379534}},
}

\bib{GG}{article}{
      author={Gabrov\v~sek, Bo\v~stjan},
      author={G\"ug\"umc\"u, Neslihan},
       title={Invariants of multi-linkoids},
        date={2023},
        ISSN={1660-5446,1660-5454},
     journal={Mediterr. J. Math.},
      volume={20},
      number={3},
       pages={Paper No. 165, 22},
         url={https://doi.org/10.1007/s00009-023-02370-w},
      review={\MR{4565101}},
}

\bib{GDS}{article}{
      author={Goundaroulis, Dimos},
      author={Dorier, Julien},
      author={Stasiak, Andrzej},
       title={Knotoids and protein structure},
        date={2020},
     journal={Topol. Geom. Biopolym},
      volume={746},
       pages={185},
}

\bib{GGLDSL}{article}{
      author={Goundaroulis, Dimos},
      author={G{\"u}g{\"u}mc{\"u}, Neslihan},
      author={Lambropoulou, Sofia},
      author={Dorier, Julien},
      author={Stasiak, Andrzej},
      author={Kauffman, Louis},
       title={Topological models for open-knotted protein chains using the
  concepts of knotoids and bonded knotoids},
        date={2017},
     journal={Polymers},
      volume={9},
      number={9},
       pages={444},
}

\bib{GK2}{article}{
      author={G\"ug\"umc\"u, N.},
      author={Kauffman, L.~H.},
       title={Parity, virtual closure and minimality of knotoids},
        date={2021},
        ISSN={0218-2165,1793-6527},
     journal={J. Knot Theory Ramifications},
      volume={30},
      number={11},
       pages={Paper No. 2150076, 28},
         url={https://doi.org/10.1142/S0218216521500760},
      review={\MR{4376721}},
}

\bib{GK1}{article}{
      author={G\"ug\"umc\"u, Neslihan},
      author={Kauffman, Louis~H.},
       title={New invariants of knotoids},
        date={2017},
        ISSN={0195-6698,1095-9971},
     journal={European J. Combin.},
      volume={65},
       pages={186\ndash 229},
         url={https://doi.org/10.1016/j.ejc.2017.06.004},
      review={\MR{3679845}},
}

\bib{GL}{article}{
      author={Gügümcü, Neslihan},
      author={Lambropoulou, Sofia},
       title={Knotoids, braidoids and applications},
        date={2017},
        ISSN={2073-8994},
     journal={Symmetry},
      volume={9},
      number={12},
         url={https://www.mdpi.com/2073-8994/9/12/315},
}

\bib{GuPf}{article}{
      author={Gügümcü, Neslihan},
      author={Pflume, Runa},
       title={Pointed quandles of linkoids},
        date={2024},
        note={arXiv:2404.14083},
}

\bib{Joyce}{article}{
      author={Joyce, David},
       title={A classifying invariant of knots, the knot quandle},
        date={1982},
        ISSN={0022-4049},
     journal={J. Pure Appl. Algebra},
      volume={23},
      number={1},
       pages={37\ndash 65},
         url={https://doi-org.ez.hamilton.edu/10.1016/0022-4049(82)90077-9},
      review={\MR{638121}},
}

\bib{Matveev}{article}{
      author={Matveev, S.~V.},
       title={Distributive groupoids in knot theory},
        date={1982},
        ISSN={0368-8666},
     journal={Mat. Sb. (N.S.)},
      volume={119(161)},
      number={1},
       pages={78\ndash 88, 160},
      review={\MR{672410}},
}

\bib{MV}{article}{
      author={Moltmaker, Wout},
      author={van~der Veen, Roland},
       title={New quantum invariants of planar knotoids},
        date={2023},
        ISSN={0010-3616,1432-0916},
     journal={Comm. Math. Phys.},
      volume={402},
      number={1},
       pages={695\ndash 722},
         url={https://doi.org/10.1007/s00220-023-04738-1},
      review={\MR{4616686}},
}

\bib{Pf}{thesis}{
      author={Pflume, Runa},
       title={Generalizations of quandles to multi-linkoids},
        type={Master's thesis},
        date={2024},
        note={http://dx.doi.org/10.53846/goediss-10442},
}

\bib{P}{article}{
      author={Polyak, Michael},
       title={Minimal generating sets of {R}eidemeister moves},
        date={2010},
        ISSN={1663-487X},
     journal={Quantum Topol.},
      volume={1},
      number={4},
       pages={399\ndash 411},
         url={https://doi-org.ez.hamilton.edu/10.4171/QT/10},
      review={\MR{2733246}},
}

\bib{Tu}{article}{
      author={Turaev, Vladimir},
       title={Knotoids},
        date={2012},
        ISSN={0030-6126},
     journal={Osaka J. Math.},
      volume={49},
      number={1},
       pages={195\ndash 223},
         url={http://projecteuclid.org/euclid.ojm/1332337244},
      review={\MR{2903260}},
}

\bib{ZhLi}{article}{
      author={Zhou, Boxin},
      author={Liu, Ximin},
       title={Quandle coloring quivers of {$(p,3)$}-torus links},
        date={2023},
        ISSN={0218-2165,1793-6527},
     journal={J. Knot Theory Ramifications},
      volume={32},
      number={3},
       pages={Paper No. 2350016, 23},
         url={https://doi.org/10.1142/S0218216523500165},
      review={\MR{4581230}},
}

\end{biblist}
\end{bibdiv}
\bibliographystyle{plain}

\end{document}